\documentclass[a4paper, 11pt,reqno]{amsart}

\usepackage[T1]{fontenc}      % Uses 'T1' font encoding.
\usepackage[foot]{amsaddr}    % Manages options for displayed author information.
\usepackage{comment}
\usepackage{siunitx}
\usepackage[british]{babel}   % Support for different languages.
\usepackage{csquotes}         % Adds quotation marks according to language.
\usepackage[babel]{microtype} % Makes things look nicer (spacing and so on).

\usepackage{mathtools}        % Extension of 'amsmath' which fixes some bugs.
\usepackage{amssymb}          % Extra symbols and fonts.
\usepackage{amsthm}           % Helps to define "theorem-like" environments.
\usepackage{amstext}          % Defines a '\text' macro for math environments.
\usepackage{mleftright}       % Fixes spacing issue.
\usepackage{gensymb}          % Allows to use degrees.
\usepackage[makeroom]{cancel} % Defines cancellation of terms in formulae.
\usepackage{mathdots}         % Adds some definitions of dots.
\usepackage{dsfont}           % Allows some fancy numbers (\mathds).
\usepackage{cancel}
\usepackage{stickstootext}
\usepackage[stickstoo,varbb]{newtxmath} % These two lines are an alternative for stix2 which do essentially the same but fix several issues with some symbols...
\makeatletter
\DeclareFontFamily{U}{ntxmia}{\skewchar \font =127}
 \DeclareFontShape{U}{ntxmia}{m}{it}{
                        <-> \ntxmath@scaled ntxmia
                      }{}    
                      \DeclareFontShape{U}{ntxmia}{b}{it}{
                        <-> \ntxmath@scaled ntxbmia
                      }{}
\makeatother

\usepackage{graphicx}         % Provides commands to include graphics.
\usepackage{psfrag}           % Gives some extra useful things for graphics.
\usepackage{caption}          % Options for formatting float captions.
\usepackage{tikz}             % Language to create graphics programmatically.
\usetikzlibrary{backgrounds}

\usepackage{booktabs}         % Enhances the quality of tables.

\usepackage{enumitem}         % Allows customization for enumerations.

\usepackage[margin=1in]{geometry} % Provides commands to define the page layout.

\usepackage[textsize=footnotesize]{todonotes}
\usepackage[numbers,sort]{natbib}

\makeatletter
\def\NAT@spacechar{~}% NEW
\makeatother

\usepackage{hyperref}              % For hyperlinks within the document
\usepackage[capitalise]{cleveref}  % For improved referencing options; use [capitalise] if desired
\hypersetup{colorlinks=true,
   citecolor=blue,
   filecolor=blue,
   linkcolor=blue,
   urlcolor=blue
}
\usepackage{url}

\makeatletter
\if@cref@capitalise
\crefname{figure}{figure}{figures}
\crefname{claim}{Claim}{Claims}
\crefname{conjecture}{Conjecture}{Conjectures}
\else
\crefname{figure}{Figure}{Figures}
\crefname{claim}{claim}{claims}
\crefname{conjecture}{conjecture}{conjectures}
\fi
\makeatother
\Crefname{figure}{Figure}{Figures}
\Crefname{claim}{Claim}{Claims}
\crefname{conjecture}{Conjecture}{Conjectures}

\usepackage{algorithm}
\usepackage{algpseudocode}

\allowdisplaybreaks

\newtheorem{theorem}{Theorem}[section]

\newtheorem{lemma}[theorem]{Lemma}
\newtheorem{corollary}[theorem]{Corollary}

\newtheorem{conjecture}[theorem]{Conjecture}

\theoremstyle{definition}
\newtheorem{definition}[theorem]{Definition}

\renewcommand{\binom}[2]{\ensuremath{\mleft(\kern-.1em\genfrac{}{}{0pt}{}{#1}{#2}\kern-.1em\mright)}}    % This makes binomial numbers nicer with stix2 (in displayed equations). Remove if stix2 is not loaded.
\newcommand{\inbinom}[2]{\ensuremath{\bigl(\kern-.1em\genfrac{}{}{0pt}{}{#1}{#2}\kern-.1em\bigr)}} % This is better for inline equations, as it will keep sizes of parentheses consistent and not create extra vertical space.

\newcommand*\nume{\ensuremath{\mathrm{e}}}

\newcommand{\cC}{\mathcal{C}}

\newcommand{\cM}{\mathcal{M}}

\newcommand{\NN}{\mathbb{N}}
\newcommand{\PP}{\mathbb{P}}
\newcommand{\EE}{\mathbb{E}}

\newcommand{\Var}{\operatorname{Var}}

\makeatletter
\def\moverlay{\mathpalette\mov@rlay}
\def\mov@rlay#1#2{\leavevmode\vtop{%
  \baselineskip\z@skip \lineskiplimit-\maxdimen
  \ialign{\hfil$\m@th#1##$\hfil\cr#2\crcr}}}
\newcommand{\charfusion}[3][\mathord]{
    #1{\ifx#1\mathop\vphantom{#2}\fi
        \mathpalette\mov@rlay{#2\cr#3}
      }
    \ifx#1\mathop\expandafter\displaylimits\fi}
\makeatother

\newcommand{\eps}{\epsilon}

\newcommand{\COMMENT}[1]{}
\newcommand{\COMNEW}[1]{}
\renewcommand{\COMNEW}[1]{\footnote{\textcolor{red!70!black}{#1}}} % comment out to hide comments

\title{How many random edges make an almost-Dirac graph Hamiltonian?}

\author[A.~Espuny D\'iaz]{Alberto Espuny D\'iaz}
\email{espuny-diaz@informatik.uni-heidelberg.de}
\address{Institut f\"ur Informatik, Universit\"at Heidelberg, 69120 Heidelberg, Germany.}
\author[R.~V.~Razafindravola]{Richarlotte Val\'er\`a Razafindravola}
\email{richarlotte.razafindravola@aims.ac.rw}
\thanks{A.~Espuny Díaz was funded by the Deutsche Forschungsgemeinschaft (DFG, German Research Foundation) through project no.\ 513704762. R.~V.~Razafindravola was funded by the DFG under Germany’s Excellence Strategy EXC-2181/1 -- 390900948 (the Heidelberg STRUCTURES Excellence Cluster).}

\date{\today}

\begin{document}
\begin{abstract}
We study Hamiltonicity in the union of an $n$-vertex graph $H$ with high minimum degree and a binomial random graph on the same vertex set.
In particular, we consider the case when $H$ has minimum degree close to $n/2$.
We determine the perturbed threshold for Hamiltonicity in this setting.

To be precise, let $\eta\coloneqq n/2-\delta(H)$.
For $\eta=\omega(1)$, we show that it suffices to add $\Theta(\eta)$ random edges to $H$ to a.a.s.\ obtain a Hamiltonian graph; for $\eta=\Theta(1)$, we show that $\omega(1)$ edges suffice.
In fact, when $\eta=o(n)$ and $\eta=\omega(1)$, we show that $(8+o(1))\eta$ random edges suffice, which is best possible up to the error term.
This determines the sharp perturbed threshold for Hamiltonicity in this range of degrees.

We also obtain analogous results for perfect matchings, showing that, in this range of degrees, the sharp perturbed thresholds for Hamiltonicity and for perfect matchings differ by a factor of $2$.
\end{abstract}

% \vspace{1em}
% Keywords: .

% \begingroup
% \def\uppercasenonmath#1{} % this disables uppercasing title
% \let\MakeUppercase\relax % this disables uppercasing authors
\maketitle
% \endgroup

\section{Introduction}\label{sec:intro}

The study of random graphs focuses on understanding the (likely) properties of the ``average'' graph (on a given probability space).
In a similar way, the study of \emph{randomly perturbed graphs} can be seen as the study of the properties of the ``average'' \emph{supergraph} of a given graph $H$.
Since the seminal work of \citet{BFM03} on Hamiltonicity of randomly perturbed graphs, these have received much attention, especially during the last decade.
Most of this research has considered supergraphs of graphs with some minimum degree condition.

More precisely, the framework we consider is the following.
Given an (integer) function $d=d(n)$, we take some sequence of $n$-vertex graphs $\{H_n\}_{n\in\mathbb{N}}$ with minimum degree $\delta(H_n)\geq d$.
We also consider the \emph{binomial random graph} $G(n,p)$ (which is an $n$-vertex graph sampled by adding each of the $\inbinom{n}{2}$ possible edges independently with probability $p$) on the same vertex set as $H_n$.
We are interested in understanding the likely properties of $H_n\cup G(n,p)$ (we only consider \emph{simple} graphs).

Much of the research into random graphs has focused on understanding the thresholds for different properties.
Roughly speaking, a \emph{threshold} refers to a range of values for the parameter~$p$ where the random graph suddenly transitions from not satisfying a property to satisfying it (with high probability).
This notion extends to the study of randomly perturbed graphs in a natural way.
Formally, in this paper we consider the following definition.

\begin{definition}
Given a nontrivial monotone graph property $\mathcal{P}$ and a function $d=d(n)$, we say that a function $p^*=p^*(n)$ is a \emph{$d$-threshold} for~$\mathcal{P}$ if the following two statements hold:
\begin{enumerate}[label=$(\arabic*)$]\setcounter{enumi}{-1}
    \item\label{0statement1} There exists a graph sequence $\{H_n\}_{n\in\mathbb{N}}$ with $\delta(H_n)\geq d$ such that, if $p=o(p^*)$, then
    \[\lim_{n\to\infty}\mathbb{P}\left[H_n\cup G(n,p)\in\mathcal{P}\right]=0.\]
    \item\label{1statement1} For every graph sequence $\{H_n\}_{n\in\mathbb{N}}$ with $\delta(H_n)\geq d$, if $p=\omega(p^*)$, we have that
    \[\lim_{n\to\infty}\mathbb{P}\left[H_n\cup G(n,p)\in\mathcal{P}\right]=1.\]
\end{enumerate}
Moreover, we say that $p^*=p^*(n)$ is a \emph{sharp $d$-threshold} for $\mathcal{P}$ if, for every fixed $\eps>0$, the following two statements hold:
\begin{enumerate}[label=$(\arabic*)$]\setcounter{enumi}{-1}
    \item\label{0statement2} There exists a graph sequence $\{H_n\}_{n\in\mathbb{N}}$ with $\delta(H_n)\geq d$ such that, if $p\leq(1-\eps)p^*$, then
    \[\lim_{n\to\infty}\mathbb{P}\left[H_n\cup G(n,p)\in\mathcal{P}\right]=0.\]
    \item\label{1statement2} For every graph sequence $\{H_n\}_{n\in\mathbb{N}}$ with $\delta(H_n)\geq d$, if $p\geq(1+\eps)p^*$, we have that
    \[\lim_{n\to\infty}\mathbb{P}\left[H_n\cup G(n,p)\in\mathcal{P}\right]=1.\]
\end{enumerate}
\end{definition}

While the (sharp) $d$-thresholds for different properties are not unique, we will follow the custom of referring to one such $d$-threshold (usually, the one with the simplest expression) as \emph{the} (sharp) $d$-threshold for~$\mathcal{P}$.
Moreover, when every sequence of graphs $\{H_n\}_{n\in\mathbb{N}}$ with $\delta(H_n)\geq d$ satisfies that $H_n\in\mathcal{P}$ (for all sufficiently large $n$), we abuse the definition and say that the (sharp) $d$-threshold for~$\mathcal{P}$ is~$0$.
Naturally, the results about thresholds in binomial random graphs correspond to $0$-thresholds.
The definitions of $d$-thresholds extend to subsequences of $n$-vertex graphs in a natural way.
Moreover, while these definitions are built on graph sequences, in practice we will omit the sequences from our statements, and they will be implicit in the use of $n$-vertex graphs $H_n$.

\subsection{Hamiltonicity}

%In this paper, we focus on the property of containing a copy of a given structure, which is clearly a monotone property.
%Let us consider Hamiltonicity (that is, the property of containing a spanning cycle) as an example for the above notions of thresholds.
In this paper, we focus mainly on Hamiltonicity (that is, the property of containing a spanning cycle).
A classical theorem of \citet{Dirac52} ensures that, if $d\geq n/2$, then the (sharp) $d$-threshold for Hamiltonicity is $0$.
%This is not the case when $d<n/2$, as can be seen by considering a complete unbalanced bipartite graph with parts of size~$d$ and~$n-d$, respectively.
On the opposite extreme, \citet{Posa76} and \citet{Kor77} independently showed that the $0$-threshold for Hamiltonicity is $\log n/n$ (in fact, \citet{Kor77} showed that $\log n/n$ is the \emph{sharp} $0$-threshold for Hamiltonicity).
Lastly, for $d=\alpha n$ with $\alpha\in(0,1/2)$ fixed, \citet{BFM03} showed that the $d$-threshold for Hamiltonicity is $1/n$.
Together, these results give a complete ``macroscopic'' picture of the thresholds for Hamiltonicity: for every $\alpha\in[0,1)$, if we set $d=\alpha n$, the $d$-threshold for Hamiltonicity is known.
However, on this macroscopic picture we observe two jumps in the behaviour of the threshold as a function of~$\alpha$.
The first such jump occurs at $\alpha=0$, and the second, at $\alpha=1/2$.
For any property $\mathcal{P}$ and any $\alpha\in[0,1]$ where such a jump occurs, we shall refer to the range $d=(1\pm o(1))\alpha n$ as a \emph{critical window}.
(This should not be confused with the probabilistic critical windows, which refer to the range around the threshold for some property.)
One should expect that, when considering these critical windows more carefully, the $d$-threshold will interpolate between the two thresholds of the macroscopic picture, to some extent.

The behaviour of Hamiltonicity around the critical window when $\alpha=0$ (which corresponds to random perturbation of sparse graphs) is fairly well understood: indeed, already \citet{BFM03} showed that, for $1\leq d=o(n)$, the function $\log(n/d)/n$ is a $d$-threshold for Hamiltonicity, and \citet{HMMMS21} improved the implicit constant to show that $p=(6+o(1))\log(n/d)/n$ suffices for $H_n\cup G(n,p)$ to a.a.s.\ contain a Hamilton cycle.
(To simplify statements, we say that an event occurs \emph{asymptotically almost surely} (a.a.s.) if the probability that it does tends to $1$ as $n$ tends to infinity.)
However, the critical window around $\alpha=1/2$ has not been studied at all.
The main goal of this note is to provide the $d$-threshold for Hamiltonicity in this critical window, thus completing the picture of the $d$-thresholds for Hamiltonicity.

\begin{theorem}\label{thm:thres}
    Let $d=n/2-\eta$, where $1/2\leq\eta\leq n/64$.
    The $d$-threshold for Hamiltonicity is~$\eta/n^2$.
\end{theorem}

We remark that, since $d$ is an integer function, the function $\eta=\eta(n)$ must take integer values when $n$ is even and ``half-integer'' values when $n$ is odd.

Our proof of \cref{thm:thres} avoids the use of the rotation-extension technique.
In the range where $\eta$ is linear, this leads to a new proof of the main result of \citet{BFM03}, in addition to the subsequent proofs of \citet{KKS16} and \citet{HMMMS21}.
Our proof does not extend to all the dense range of degrees (the upper bound on $\eta$ in the statement is not best possible but, since we cannot hope to get close to $n/2$ with our approach, we have made no effort to optimise it).
However, for the range that we consider, our proof yields a better upper bound on the sharp $d$-threshold for Hamiltonicity than that in~\cite{HMMMS21}, and thus improves on all previously known bounds (see \cref{2ConnectedGraph}~\ref{2ConnectedGraphitem2}).
In fact, for a smaller range of the function $\eta$ (but covering essentially all the critical window), our technique provides us with the \emph{sharp} $d$-threshold for Hamiltonicity.

\begin{theorem}\label{thm:sharp}
    Let $d=n/2-\eta$, where $\eta=\omega(1)$ and $\eta=o(n)$.
    The sharp $d$-threshold for Hamiltonicity is~$16\eta/n^2$.
\end{theorem}

This is the first (non-trivial) result about sharp $d$-thresholds with $d=\omega(1)$ in the literature.
We remark that, when $\eta=\Theta(1)$, there is no sharp $d$-threshold for Hamiltonicity, so the lower bound on $\eta$ in \cref{thm:sharp} is necessary.
The extremal graph witnessing the \ref{0statement2}-statement is the standard one for the Hamiltonicity problem: a complete bipartite graph $H_n$ with parts of size $d$ and $n-d$, respectively; see \cref{sec:proofs} for more details.

As a corollary, we also obtain the sharp $d$-threshold for pancyclicity (that is, the property of containing a cycle of every length between $3$ and $n$) for a slightly more restricted range of~$d$.
This property has been considered in randomly perturbed graphs by \citet{KKS16} as well as \citet{AE22}.
The following is an immediate consequence of \cref{thm:sharp} and~\cite[Theorem~4]{AE22}.

\begin{corollary}\label{coro:pan}
    Let $d=n/2-\eta$, where $\eta=\omega(\log n)$ and $\eta=o(n)$.
    The sharp $d$-threshold for pancyclicity is~$16\eta/n^2$.\COMMENT{Proof: The \ref{0statement2}-statement is trivial, as it is also a bound for Hamiltonicity, so we focus on the \ref{1statement2}-statement. Let $H_n$ be an arbitrary $n$-vertex graph with $\delta(H_n)\geq d$. By \cref{thm:sharp}, we may choose $p_1=(16+o(1))\eta/n^2$ such that a.a.s.\ $H_n\cup G(n,p_1)$ is Hamiltonian. Conditional on this, by \cite[Theorem~4]{AE22}, we may choose $p_2=\Theta(\lg n/n^2)$ such that a.a.s.\ $H_n\cup G(n,p_1)\cup G(n,p_2)$ is pancyclic. By the choice of $\eta$, we have that $p_1+p_2=(16+o(1))\eta/n^2$.}
\end{corollary}

\subsection{Perfect matchings}

Let us now assume that $n$ is an even integer (and so all the sequences of graphs are restricted to even values of $n$).
The property of containing a perfect matching (that is, a set of $n/2$ pairwise disjoint edges) is one of the most studied in graph theory.
Since a Hamilton cycle contains a perfect matching, Dirac's theorem~\cite{Dirac52} ensures that, for $d\geq n/2$, the (sharp) $d$-threshold for containing a perfect matching is $0$. % (and the lower bound on $d$ is again best possible).
A classical result of \citet{ER66} shows that the sharp $0$-threshold for perfect matchings is $\log n/n$.
Note that these coincide with the respective sharp $d$-thresholds for Hamiltonicity (in the case of the $0$-threshold, a difference can be seen by analysing smaller order terms).
For randomly perturbed graphs with $1\leq d<n/2$, the results of \citet{BFM03}, together with \cref{thm:thres} and the standard extremal graph, immediately imply that the $d$-threshold for perfect matchings also coincides with the $d$-threshold for Hamiltonicity.
However, it turns out that the \emph{sharp} $d$-threshold for perfect matchings does \emph{not} coincide with that for Hamiltonicity.
Our techniques allow us to determine this sharp threshold too, in the same range of $d$ as for Hamiltonicity.

\begin{theorem}\label{thm:PM}
    Let $d=n/2-\eta$, where $\eta=\omega(1)$ and $\eta=o(n)$.
    The sharp $d$-threshold for containing a perfect matching is $8\eta/n^2$.
\end{theorem}

The extremal graph witnessing the \ref{0statement2}-statement for this sharp threshold is the same standard extremal example for this problem: a complete bipartite graph $H_n$ with parts of size $d$ and $n-d$, respectively (see \cref{sec:proofs} for the details).
Our proof of \cref{thm:PM} also works when $n$ is odd to show that the sharp $d$-threshold for containing a matching of size $\lfloor n/2\rfloor$ is $8\eta/n^2$.

\subsection{Related work}

Hamiltonicity is likely the property that has been studied the most in randomly perturbed graphs.
In addition to the papers we already mentioned (for graphs), it has also been studied in directed graphs~\cite{BFM03,KKS16,ABKPT24}, hypergraphs~\cite{KKS16,MM18,HZ20} or subgraphs of the hypercube~\cite{CEGKO20}. %See also BHKM19 and CHT23 for powers of cycles in hypergraphs.
On graphs, Hamiltonicity has also been considered when the perturbation comes from a random regular graph~\cite{EG23} or a random geometric graph~\cite{E23,EH24}.
A variant of the problem considers random colourings of randomly perturbed graphs and aims to find rainbow Hamilton cycles~\cite{AF19,AH21,KLS23}.
The Hamiltonicity Maker-Breaker~\cite{CHMP21a} and Waiter-Client~\cite{CHMP21b} games played on randomly perturbed graphs have been considered as well.

In addition to Hamiltonicity, multiple other spanning properties have been studied in the context of randomly perturbed graphs.
For instance, the full ``macroscopic'' behaviour of the $d$-thresholds is known for connectivity~\cite{BFKM04}, spanning bounded degree trees~\cite{KKS17,BHKMPP19} and (almost) unbounded degree trees~\cite{JK19} (for which, analogously to Hamiltonicity, there are two critical windows around $\alpha=0$ and $\alpha=1/2$); triangle-factors~\cite{BTW19,HMT21,BPSS23} and $2$-universality~\cite{Parc20,BPSS24} (for which there are three critical windows, around $\alpha=0$, $\alpha=1/3$ and $\alpha=2/3$); $C_\ell$-factors~\cite{BTW19,BPSS21} (for which there also are three critical windows, which depend on the value of $\ell$), and squares of Hamilton cycles~\cite{DRRS20,BPSS24} (for which there is an infinite number of critical windows, one at $\alpha=2/3$ and one around $\alpha=1/k$ for each integer $k\geq3$).
For $K_r$-factors with $r\geq4$, the ``macroscopic'' $d$-threshold is known for all but finitely many points~\cite{BTW19,HMT21,AKR24}, which in particular shows that there are $r$ critical windows (at $\alpha=k/r$ for all $k\in\{0,\ldots,r-1\}$).
We believe that studying the $d$-thresholds for these properties in the respective critical windows would be a very interesting problem.
Additionally, for some properties, the $d$-thresholds are known for some range of the values of $d$.
For example, there are results about different $F$-factors~\cite{BTW19} and higher powers of Hamilton cycles~\cite{BMPP20,DRRS20,ADRRS21,NT21,ADR23,NP24}, as well as some results for general bounded degree spanning graphs~\cite{BMPP20}.

\section{Proofs}\label{sec:proofs}

We begin this section by proving the \ref{0statement1}-statement for our different results.
As already mentioned in the introduction, these follow from a construction that is standard in the area; we include the details for the benefit of the unfamiliar reader.
Let $d=n/2-\eta$, where $1/2\leq\eta=\eta(n)\leq n/64$.
Let $H_n$ be a complete bipartite graph with parts~$A$ and $B$ of size $d$ and $n-d$, respectively.

Suppose first that we wish to obtain a graph which contains a matching of size $\lfloor n/2\rfloor$ on the vertex set $A\cup B$.
As $|A|<|B|$, any such matching must have at least $\lfloor n/2\rfloor-|A|\geq\eta-1/2$ edges contained in $B$.
The expected number of edges in $B$ in the random graph $G(n,p)$ is $\EE[|E(G(n,p)[B])|]=\inbinom{n/2+\eta}{2}p=\Theta(n^2p)$.
Thus, by Markov's inequality, if $p=o(\eta/n^2)$, then a.a.s.\ $G(n,p)[B]$ contains $o(\eta)$ edges.
Since any Hamilton cycle contains a matching of size $\lfloor n/2\rfloor$, this completes the proof of the \ref{0statement1}-statement for \cref{thm:thres}.
(There is in fact one missing case, when $n$ is odd and $\eta=1/2$; in this case, clearly any Hamilton cycle must contain at least one edge in $B$, and this will not occur if $p=o(n^{-2})$, again by Markov's inequality.)

Suppose now that $\eta=\omega(1)$ and $\eta=o(n)$.
Then, $\EE[|E(G(n,p)[B])|]=\inbinom{n/2+\eta}{2}p=(1+o(1))n^2p/8$.
Since $|E(G(n,p)[B])|$ is a binomial random variable, for any fixed $\eps>0$, if $p\leq(1-\eps)8\eta/n^2$, it follows by Chernoff's inequality\COMMENT{$\PP[X\ge (1+\gamma)\EE[X]]=\PP[X\ge2\eta(1-\epsilon^2)]\le 2\exp({-\frac{\epsilon^2}{3}(1-\epsilon)\eta})\underset{n\rightarrow \infty}{\rightarrow 0}$. It is also good to note that  $\PP[X\ge 2\eta]\le\PP[X\ge 2\eta(1-\epsilon^2)] $.} that a.a.s.\ $|E(G(n,p)[B])|\leq(1-\eps/2)\eta$, and so $H_n\cup G(n,p)$ does not contain a matching of size $\lfloor n/2\rfloor$. 
This completes the \ref{0statement2}-statement for \cref{thm:PM}.

Lastly, consider any cycle on vertex set $A\cup B$.
Each such cycle can be mapped to a (cyclic) word consisting of $n$ symbols, each being an $A$ or a $B$ depending on which set each vertex of the cycle belongs to.
Due to the difference in sizes of the sets, each such word must have at least $|B|-|A|=2\eta$ pairs of consecutive $B$'s.
In other words, in order for $H_n\cup G(n,p)$ to contain a Hamilton cycle, a necessary condition is that $G(n,p)[B]$ must contain at least $2\eta$ edges.
Arguing like above with Chernoff's inequality, we conclude that a.a.s.\ this does not hold if $p\leq(1-\eps)16\eta/n^2$, thus completing the \ref{0statement2}-statement for \cref{thm:sharp}.

From now on, we focus on the proofs of the \ref{1statement1}-statements.
For technical reasons, it will be convenient to step away from the binomial random graph model and instead consider uniform random graphs.
Given an integer $m\in[\inbinom{n}{2}]$, the random graph $G_{n,m}$ is an $n$-vertex graph with exactly $m$ edges chosen uniformly at random among all such graphs.
A well-known coupling argument allows us to work with this model and transfer the results we obtain to the binomial random graph model.
Indeed, assuming $m=\omega(1)$, there exist $p=(1+o(1))m/\inbinom{n}{2}$ and a coupling $(G_1,G_2)$ such that $G_1\sim G_{n,m}$, $G_2\sim G(n,p)$ and a.a.s.\ $G_1\subseteq G_2$.
As such, in order to conclude our proofs, it suffices to study $H_n\cup G_{n,m}$.

\subsection{Hamiltonicity}

Given that the \ref{0statement1}-statements have already been proved, \cref{thm:thres,thm:sharp} are an immediate consequence of the following theorem. 

\begin{theorem}\label{2ConnectedGraph}
Let $H_n$ be a graph on $n$ vertices with $\delta(H_n)\ge n/2-\eta$, where $1/2\le \eta\le n/64$.
\begin{enumerate}[label=$(\alph*)$]
\item\label{2ConnectedGraphitem1} If $m=\omega(\eta)$, then a.a.s.\ $H_n\cup G_{n,m}$ is Hamiltonian.
\item\label{2ConnectedGraphitem2} If $\eta=\omega(1)$, $\lambda=\omega(\eta^{1/2})$ and $m\geq\frac{2\eta}{1/4-8\eta/n}+\lambda$, then a.a.s.\ $H_n\cup G_{n,m}$ is Hamiltomian.
\end{enumerate}
\end{theorem}

Roughly speaking, the idea of the proof is the following.
The minimum degree of the graphs~$H_n$ that we consider ensures that, if they are $2$-connected, then they already contain an almost spanning cycle (see \cref{DiracLemma}).
An analysis of the properties of such $H_n$ shows that, if they are not Hamiltonian, then there are \emph{many} non-edges that, if added to the graph, would result in a graph containing a longer cycle than $H_n$ (essentially as many as in the extremal example, so this can be regarded as a stability version of Dirac's theorem; see \cref{RandomEdges}).
Now, given any $H_n$ with the desired minimum degree, by sprinkling a few random edges, we can ensure that a.a.s.\ it is $2$-connected (using work of \citet{BFKM04}, see \cref{BKFM}).
Then, in rounds, we will add random edges to the graph until a longer cycle can be found.
This is iterated until a Hamilton cycle appears.
Since the starting graph is almost Hamiltonian, the number of rounds that is needed is small. 
Moreover, since at every step there is a large set of ``good'' non-edges, each round is likely to finish rather quickly.
This leads to the improved bounds.

One of the main tools for our proof is the following classical result of Dirac, which ensures that~$H_n$ contains an almost spanning cycle.

\begin{lemma}[{\citet[Theorem 4]{Dirac52}}]\label{DiracLemma}
Any $2$-connected graph $G$ on $n$ vertices with $\delta(G)=d$, where $1< d\le n/2$, contains a cycle of length at least $2d$.
\end{lemma}

The next lemma ensures that, for any $2$-connected non-Hamiltonian graph $H_n$ with high minimum degree, there is a large number of non-edges (essentially as many as in an unbalanced complete bipartite graph) which, if added to $H_n$, would result in a graph containing longer cycles than $H_n$.

\begin{lemma} \label{RandomEdges}
Let $H_n$ be an $n$-vertex $2$-connected graph with $\delta(H_n)\ge n/2-\eta>1$.
Suppose $H_n$ is not Hamiltonian and let $\cC$ be a longest cycle in $H_n$.
Then, there exists a set $E\subseteq  E(K_n)$ of size $|E|\ge n^2/8-4n\eta$ such that, for any $e\in E$, $H_n\cup \{e\}$ contains a cycle longer than $\cC$.\COMMENT{In fact, for every $v\in V(H_n)$, there is a set $E_v$ of size $|E_v|\ge n^2/8-4n\eta$ such that, for any $e\in E_v$, $H_n\cup \{e\}$ contains a cycle which is longer than $\cC$ containing $v$.}
\end{lemma}

\begin{proof}
For each $v\in V(H_n)$, let $N_{\cC}(v)\coloneqq\{w\in V(\cC)\mid vw\in E(H_n)\}$. By \cref{DiracLemma}, the cycle $\cC$ has length at least $n-2\eta$.   From the definition of $N_{\cC}(v)$ and the minimum degree condition, for any $v\in V(H_n)$, we have that
\begin{equation}\label{MinDegCondt}
|N_{\cC}(v)|\ge n/2-3\eta.
\end{equation}

Let us now fix an orientation of the cycle $\cC$ and, for each vertex $w\in V(\cC)$, denote by $w^-$ and $w^+$  the predecessor and the successor of  $w$ in the orientation of $\cC$, respectively. Then, fix an arbitrary vertex $v\in V(H_n)\setminus V(\cC)$. Note that, if $w\in N_{\cC}(v)$, then $w^-\notin N_{\cC}(v)$, as otherwise $\cC\cup\{ w^-v,vw\}\setminus \{w^-w\}$  would form a cycle longer than $\cC$. Analogously, $w^+\notin N_{\cC}(v)$.  From this fact, one can easily see that  $N_{\cC}(v)$ decomposes $\cC$ into paths of length at least $2$. For each $i\ge2$, denote by $X_i$ the number of paths of length $i$ in $\cC$ which result from this decomposition of $\cC$, so we have that
\begin{equation*}
\underset{i\ge2}{\sum}X_{i}=|N_{\cC}(v)|\ge n/2-3\eta.
\end{equation*}
It follows  that 
\begin{equation*}
|V(\cC)|=\underset{i\ge2}{\sum} iX_i \ge 2X_2+3\left(\underset{i\ge2}{\sum}X_{i}-X_2\right)\ge 2X_2+3\left(\frac{n}{2}-3\eta-X_2\right),
\end{equation*}
which implies that 
\begin{equation} \label{C-4bound}
X_2\ge \frac{n}{2}-9\eta.
\end{equation}

Let $W_v\coloneqq \{w\in V(\cC)\setminus N_{\cC}(v)\mid w^-,w^+\in N_{\cC}(v)\}$ and note that $|W_v|=X_2$.
Observe that,  if $w\in W_v$ and $u\in N_{\cC}(w)\setminus\{w^+\}$, then $u^-\notin N_{\cC}(w)$, as otherwise 
\[\cC\setminus \{w^-w,ww^+,u^-u\}\cup\{ w^-v,vw^+\}\cup\{u^-w,wu\}\]
would form a cycle longer than $\cC$.
In other words, for every $w\in W_v$ and $u\in N_{\cC}(w)\setminus\{w^+\}$, we have that $wu^-\notin E(H_n)$ and $H_n\cup\{wu^-\}$ would have a cycle longer than $\cC$.
Thus, using \eqref{MinDegCondt} and \eqref{C-4bound}, there exist at least $ (n/2-9\eta)(n/2-3\eta-1)/2\ge n^2/8-4n\eta$  edges $e\in E(K_n)\setminus E(H_n)$ such that $H_n\cup \{e\}$ contains a longer cycle than $\cC$ (where we divide by $2$ to avoid double counting  edges).
\end{proof}
%%%%%%%%%%%%%%%%%
Our last tool is a particular case of a result of \citet{BFKM04} which ensures that, if $H_n$ has high minimum degree, then a.a.s.\ we need to add very few random edges for it to become $2$-connected.

\begin{lemma}[{\citet[Theorem~6]{BFKM04}}]\label{BKFM}
Let $H_n$ be an $n$-vertex graph with $\delta(H_n)\ge\alpha n$, where $\alpha\in(0,1)$ is fixed. If $m=\omega(1)$, then a.a.s.\ $H_n\cup G_{n,m}$ is $2$-connected.
\end{lemma}
%%%%%%%%%%%%%%%%%%%%%%%
With these results in hand, we can already prove \cref{2ConnectedGraph}.

\begin{proof}[Proof of \cref{2ConnectedGraph}]

%\noindent By \cref{DiracLemma} $H_n$ has a cycle $\cC$ of length at least $n-2\eta$. 
First, let $m_0=\omega(1)$ grow arbitrarily slowly. Let $G_0\coloneqq H_n\cup G_{n,m_0}$. By \cref{BKFM}, a.a.s.\ $G_0$ is $2$-connected. Condition on this event. Let $\cC$ be a longest cycle in $G_0$ and  $k \coloneqq |V(H_n)\setminus V(\cC)|$. Note that $k\le 2\eta$ by \cref{DiracLemma}. Let $E_0$ be the set of all edges 
$e\in E(K_n)\setminus E(G_0)$ such that $G_0\cup\{e\}$ has a longer cycle than $G_0$.

Now, consider a sequence of random graphs $G_0\subseteq G_1\subseteq G_2\subseteq \dots\subseteq G_k$ defined as follows. For each $i\in [k]$ , if $ G_{i-1}^0\coloneqq G_{i-1}$  is not Hamiltonian, we start sampling edges $e^{(1)}_i, e^{(2)}_i,\ldots\in E(K_n)$ uniformly at random with replacement and, for each $j\in \NN$, define $G_{i-1}^j\coloneqq G_{i-1}^{j-1}\cup\{e_i^{(j)}\}$. In this case, we let $Y_i$ be the minimum $j\in \NN$ such that $e_i^{(j)}\in E_{i-1}$ and, then, set $G_i\coloneqq G_{i-1}^{Y_i}$. Moreover, we define $E_i$ to be the set of all edges $e\in E(K_n)\setminus E(G_i)$ such that $G_i\cup\{e\}$ has a longer cycle than $G_i$. If instead $G_{i-1}$ is Hamiltonian, we let $G_i=G_{i+1}=\dots=G_k\coloneqq G_{i-1}$ and $Y_i=Y_{i+1}=\dots=Y_k\coloneqq0$. 

Let $h$ be the minimum $i\in [k]$ for which $G_i$ is Hamiltonian.\COMMENT{Note that $h$ is a random variable.} Note that this is well defined by the definition of $k$, since each subsequent $G_i$ either is Hamiltonian or contains a longer cycle than $G_{i-1}$. By \cref{RandomEdges}, the random variables $Y_1,Y_2,\dots,Y_h$ are independent geometric random variables of parameter $p_i={|E_{i-1}|}/{\inbinom{n}{2}}\ge1/4-8\eta/n$.  We set $Y\coloneqq\sum_{i=1}^kY_i$ to be the total number of random edges added until $G_k$ is Hamiltonian, and let $G^Y$ denote the $n$-vertex graph consisting of all these random edges. Note that there is a trivial coupling between $G_{n,m_0}$ and $G_{n,m_0+Y}$ such that $G_{n,m_0}\cup G^Y\subseteq G_{n,m_0+Y}$, so it suffices to bound $Y$ to reach the desired conclusion.

In order to prove \ref{2ConnectedGraphitem1}, note that  $\EE[Y]\le 8k\le16\eta$. Thus, by Markov's inequality\COMMENT{ $\PP(Y\ge m)\le \frac{8\eta(1+o(\eta/n))}{m}$} and since $m_0$ is chosen to grow arbitrarily slowly, if $m=\omega(\eta)$, then a.a.s.\ $Y\le m$, that is, a.a.s.\ adding  $m_0+m=\omega(\eta)$ random edges to  $H_n$ suffices to obtain a  Hamiltonian graph.

%We now define a sequence of random graphs $G_0\subseteq G_1\subseteq G_2\subseteq \dots\subseteq G_k$ where, for each $i\in [k]$ and by \cref{DiracLemma} and \cref{RandomEdes}, $G_i$ contains a cycle $\cC_i$ of length at least $n-2\eta$ and  there exists a set $E_i$  of potential edges, with $|E_i|\ge n^2/8-4n\eta\ge n^2/16$, such that for each $e\in E_i$, $G_{i-1}\cup\{e\}$ contains a  cycle longer than $\cC_{i-1}$. 

In order to prove \ref{2ConnectedGraphitem2}, assume that $\eta=\omega(1)$. We have that $\EE[Y]\le \frac{2\eta}{1/4-8\eta/n}$ and $\Var (Y)\le 112\eta$.\COMMENT{$\Var(Y)\le\sum_{i=1}^k\Var(Y_i)\le2\eta\frac{3/4+8\eta/n}{(1/4-8\eta/n)^2}\le 2\eta\frac{7/8}{(1/8)^2}\le112\eta$} By Chebyshev's inequality, for any $\lambda=\omega(\eta^{1/2})$, we have that $\PP[| Y-\EE[Y]|\ge \lambda/2]=o(1)$, so a.a.s.\ $Y\le \EE[Y]+\lambda/2\le \frac{2\eta}{1/4-8\eta/n}+\lambda/2$. Since $m_0$ is chosen to grow arbitrarily slowly, we may take $m_0\le\lambda/2$, and it follows that a.a.s.\ adding $m_0+Y\le \frac{2\eta}{1/4-8\eta/n}+\lambda$ random edges to $H_n$ suffices to obtain a Hamiltonian graph.
\end{proof}

\subsection{Perfect matchings}

For simplicity, throughout this section we assume that $n$ is even.
Our goal is to complete the proof of \cref{thm:PM}.
Since we already showed its \ref{0statement2}-statement, it suffices to prove the following result.
(For the sake of simplifying the calculations, we state it only for $\eta=o(n)$, which suffices for \cref{thm:PM}; for larger values of $\eta$, a statement similar to \cref{2ConnectedGraph}~\ref{2ConnectedGraphitem2} holds as well.)

\begin{theorem}\label{PMatching}
Let $H_n$ be a graph on $n$ vertices with minimum degree at least $n/2-\eta$. 
If $\eta=o(n)$ with $\eta=\omega(1)$, $\lambda=\omega(\eta^{1/2})$ and $m\ge 4\eta+\lambda$, then a.a.s.\ $H_n\cup G_{n,m}$ contains a perfect matching.
\end{theorem}

The proof of \cref{PMatching} follows along similar lines as that of \cref{2ConnectedGraph}.
The main part is to show that, if $H_n$ is $2$-connected, then there are \emph{many} non-edges that, if added to $H_n$, would result in a graph with a larger matching than $H_n$.

\begin{lemma}\label{MAtchingEdges}
Let $H_n$ be an $n$-vertex $2$-connected graph  with $\delta(H_n)\ge n/2-\eta>1$. Assume that $H_n$ does not contain a perfect matching and let $\cM$ be a largest matching in $H_n$. Then, there exists a set $M\subseteq E(K_n)$ of size at least $n^2/8-4\eta n$ such that, for any $e\in M$, $H_n\cup \{e\}$ contains a matching which is larger than $\cM$.\COMMENT{In fact, for any pair of vertices $u,v\in V(H_n)\setminus V(\cM)$, there exists a set $M_{uv}$ such that, for any $e\in M_{uv}$, $H_n\cup \{e\}$ contains a larger matching containing $u$ and $v$.}
\end{lemma}

\begin{proof}
By \cref{DiracLemma}, $H_n$ contains a cycle of length at least $n-2\eta$, so $H_n$ contains a matching of size at least $n/2-\eta$. Let $\cM$ be a largest matching in $H_n$ and, for each $w\in V(\cM)$, let $w^{\cM}$ denote the vertex $w'\in V(\cM)$ such that $ww'\in \cM$. Moreover, for each $u\in V(H_n)\setminus V(\cM)$, let $N_{\cM}(u)\coloneqq\{w\in V(\cM) \mid  w^{\cM}\in N(u)\}$, where $N(u)$ represents the set of all neighbours of~$u$ in~$H_n$. Note that, for any pair of distinct vertices $u,v\in V(H_n)\setminus V(\cM)$, $H_n$ cannot contain a $(u,v)$-path of length $5$ using two edges of $\cM$, as otherwise this would be an $\cM$-augmenting path.  It follows that, for any  pair of distinct $u,v\in V(H_n)\setminus V(\cM)$ and all $x\in N_{\cM}(u)$ and $y\in N_{\cM}(v)$ with $x\neq y$, we must have that $e=xy\notin E(H_n)$, since otherwise we would have one such  path of length $5$. In particular, each such edge is a potential edge that, if added to $H_n$, creates a graph with a larger matching than $\cM$. From the minimum degree condition and the size of $\cM$, $|N_{\cM}(u)|\ge n/2-3\eta$ and $|N_{\cM}(v)|\ge n/2-3\eta$. Hence, there exist at least $\inbinom{n/2-3\eta}{2}\ge n^2/8-4\eta n$ such potential edges.
\end{proof}

The proof of \cref{PMatching} is now essentially the same as the proof of \cref{2ConnectedGraph}~\ref{2ConnectedGraphitem2}, using \cref{MAtchingEdges} instead of \cref{RandomEdges}.
The main difference is that, through the sequence of random graphs, each subsequent largest matching contains at least two more vertices than the previous, and so one may take $k\coloneqq|V(H_n)\setminus V(\mathcal{M})|/2\leq\eta$ (where $\mathcal{M}$ is a largest matching in $G_0$).
This smaller number of random graphs in the sequence leads to the improved bound on $m$.
We omit the details of the proof.

\section{Open problems}

With \cref{thm:sharp,thm:PM}, we have showed that, if $d=n/2-o(n)$, the sharp $d$-thresholds for Hamiltonicity and perfect matchings coincide with what is needed for a randomly perturbed unbalanced complete bipartite graph to contain a Hamilton cycle or perfect matching.
It seems plausible that unbalanced complete bipartite graphs should also witness the sharp $d$-threshold for perfect matchings for smaller values of $d$.
Let $d=\alpha n$ for some fixed constant $\alpha\in(0,1/2)$, and let $H_n$ be a complete bipartite graph with parts $A$ and $B$ of size $d$ and $n-d$, respectively.
Note that $H_n$ contains a matching of size $d$, where, moreover, its $d$ vertices in $B$ can be chosen arbitrarily.
We know, from the work of \citet{BFM03} and the extremal example of the complete bipartite graph, that the sharp $d$-threshold for the containment of a perfect matching (if it exists) must be of the form $C/n$, where $C=C(\alpha)$ is a constant. 
We conjecture that this threshold should coincide with the threshold for $G(n,p)[B]\sim G(n-d,p)$ to contain a matching of size $n/2-d$ (which can then be completed using edges of $H_n$).
The (likely) size of a largest matching in sparse random graphs was determined by \citet{KS81} (see also \cite[Theorem~4]{AFP98} for a concrete expression).
Using their work, we propose the following conjecture.

\begin{conjecture}\label{conj:PM}
    Let $\alpha\in(0,1/2)$ be fixed, and let $d=\alpha n$.
    The sharp $d$-threshold for containing a perfect matching is $C/n$, where $C=C(\alpha)$ is the solution to the equation
    \[1-\frac{\gamma_*+\gamma^*+\gamma_*\gamma^*}{(2-2\alpha)C}=\frac{1-2\alpha}{2-2\alpha},\]
    where $\gamma_*$ is the smallest root of the equation $x=(1-\alpha)C\exp(-(1-\alpha)C\nume^{-x})$ and $\gamma^*=(1-\alpha)C\nume^{-\gamma_*}$.\COMMENT{The result of \citet{KS81} (see \cite[Theorem~4]{AFP98}) states that, for $p=C/n$, a.a.s.\ the largest matching of $G(n,p)$ has size 
    \[(1\pm o(1))\left(1-\frac{\theta_*+\theta^*+\theta_*\theta^*}{2C}\right)n,\]
    where $\theta_*$ is the smallest root of the equation $x=C\exp(-C\nume^{-x})$ and $\theta^*=C\nume^{-\theta_*}$.
    Naturally, we are interested in the maximum matching in $G((1-\alpha)n,p)\sim G(n',(1-\alpha)C/n')$ as $n'$ tends to infinity.
    This yields the expression on the left above.
    And we wish this matching to have size at least $n/2-d=(1-2\alpha)n/2=\frac{1-2\alpha}{2(1-\alpha)}n'$.}
\end{conjecture}

It also seems plausible that complete bipartite graphs should be the extremal example for Hamiltonicity for the entire range of $d$.
Clearly, the complete bipartite graph $H_n$ defined above contains a cycle of length $2d$, and no longer cycles.
If an edge with both endpoints in $B$ is added to $H_n$, this can be used to construct a longer cycle.
In this case, however, it is not only isolated edges that are useful for constructing longer cycles: indeed, any path (of length at most $n-2d$) contained in~$B$ can be incorporated into a cycle.
Thus, a linear forest (that is, a collection of vertex-disjoint paths) containing $n-2d$ edges within $B$ can be used to construct a Hamilton cycle.

\begin{conjecture}\label{conj:Ham}
    Let $\alpha\in(0,1/2)$ be fixed, and let $d=\alpha n$.
    The sharp $d$-threshold for Hamiltonicity coincides with the sharp threshold for $G(n-d,p)$ to contain a linear forest of size $n-2d$.
\end{conjecture}

In this case, we do not propose an explicit expression for the threshold since the size of the largest linear forest in a sparse random graph has not been considered in the literature.
Studying this problem may be of independent interest.
We expect the sharp $d$-threshold for Hamiltonicity to differ from that for perfect matchings by a constant factor, with this constant factor depending on $\alpha$ and tending to $1$ as $\alpha$ tends to $0$.

We also believe \cref{conj:Ham} should hold for pancyclicity.
Moreover, we have no reason to believe that the lower bound on $d$ in \cref{coro:pan} is necessary, and think that a statement analogous to \cref{thm:sharp} should hold for pancyclicity as well.
Naturally, we also believe that determining the sharp $d$-threshold for perfect matchings, Hamiltonicity and pancyclicity in the critical window $d=o(n)$ is a problem of interest.
%We suspect these sharp thresholds may coincide at $\log(n/d)/n$.

Lastly, we want to consider the extension of our results to randomly perturbed \emph{directed} graphs (or \emph{digraphs} for short), where we allow up to two edges between each pair of vertices, one in each direction.
We define thresholds analogously as above, where now the binomial random digraph $D(n,p)$ is obtained by adding each of the possible $n(n-1)$ edges independently at random with probability $p$, and instead of the minimum degree of a graph we consider the minimum \emph{semidegree} of the digraph, which is the minimum, over all vertices, of the minimum between the number of edges leaving and the number of edges arriving at each vertex.
A classical result of \citet{GH60} shows that, if $\delta^0(D)=d\geq n/2$, then the (sharp) \mbox{$d$-threshold} for Hamiltonicity is $0$.
The sharp $0$-threshold is again $\log n/n$ (as follows by a general coupling argument of \citet{Mc83}).
The Hamiltonicity of randomly perturbed digraphs when $d=\alpha n$ with $\alpha\in(0,1/2)$ fixed was studied by \citet{BFM03}, who showed that the $d$-threshold in this case is also $1/n$ (later, \citet{KKS16} provided a new proof of this fact).
Just like in graphs, the thresholds present two critical windows around $\alpha=0$ and $\alpha=1/2$, and it is thus natural to consider the $d$-thresholds in these regimes.
The extension of \cref{thm:thres} to digraphs remains open.

\begin{conjecture}\label{conj:digraph}
    Let $d=n/2-\eta$, where $1/2\leq\eta=\eta(n)=o(n)$.
    The $d$-threshold for Hamiltonicity in randomly perturbed directed graphs is $\eta/n^2$.
\end{conjecture}

Very recently, \citet{ABKPT24} considered different orientations of Hamilton cycles in randomly perturbed digraphs.
It would also be interesting to extend their work to the corresponding critical windows. %, and we believe \cref{conj:digraph} should hold for other orientations of a Hamilton cycle.

% Use with natbib, not biblatex:
\bibliographystyle{mystyle} 
\bibliography{references}
 
% Use with biblatex, not natbib:
% \printbibliography

\end{document}